\PassOptionsToPackage{dvipsnames}{xcolor}
\documentclass[review,leqno]{siamart1116}
\usepackage{amssymb}
\usepackage{amsfonts}
\usepackage{amsmath}
\usepackage{graphicx}
\usepackage{xspace}
\usepackage{multirow}
\usepackage{siunitx}
\usepackage{dutchcal}
\usepackage{derivative}
\usepackage{bm,bbm}
\usepackage{mydef}
\usepackage[numbers]{natbib}
\let\cite=\citet
\usepackage{enumerate}
\usepackage{dsfont}
\usepackage{enumitem}
\usepackage[algo2e,linesnumbered]{algorithm2e}
\usepackage{subcaption}
\usepackage{verbatim}
\usepackage[normalem]{ulem}
\usepackage{longtable}
\usepackage{tikz}
\usepackage{pgfplots}

\usepackage{caption}
\usepackage{hyperref}  
\usepackage{cleveref}          
\pgfplotsset{compat=1.18}
\nolinenumbers

\begin{document}

\newcommand\footnotemarkfromtitle[1]{%
	\renewcommand{\thefootnote}{\fnsymbol{footnote}}%
	\footnotemark[#1]%
	\renewcommand{\thefootnote}{\arabic{footnote}}}

\newcommand{\TheTitle}{A High Order IMEX Method for Generalized Korteweg de-Vries Equations}
\newcommand{\TheAuthors}{S. Gerberding}

\headers{IMEX Methods for Generalized Korteweg de-Vries Equations}{\TheAuthors}

\title{{\TheTitle}\thanks{Draft version, \today \funding{This material is based upon work supported in part by
   the National Science Foundation grant DMS2110868, the Air Force
   Office of Scientific Research, USAF, under grant/contract number
   FA9550-18-1-0397, the Army Research Office, under grant number
   W911NF-19-1-0431.}}}

\author{
	Seth Gerberding\footnotemark[1]
}

\maketitle

\renewcommand{\thefootnote}{\fnsymbol{footnote}}
\footnotetext[1]{Department of Mathematics, Texas A\&M University 3368
	TAMU, College Station, TX 77843, USA.}

\renewcommand{\thefootnote}{\arabic{footnote}}
\begin{abstract}
In this paper, we introduce a high order space-time approximation of generalized Korteweg de-Vries equations. More specifically, the method uses continuous \(H^1\)-conforming finite elements for the spatial approximation and implicit-explicit methods for the temporal approximation. The method is high order in both space, provably stable, and mass-conservative. The scheme is formulated, its properties are proven, and numerical simulations are provided to illustrate the proposed methodology.
\end{abstract}

\begin{keywords}
    implicit-explicit time integration methods, dispersive equations, \\ high-order method, mass conservative
\end{keywords}

\begin{AMS}
	65M60, 65M12, 35G31, 35G20, 37L65
\end{AMS}

\section{Introduction}\label{sec:introduction}
This paper is concerned with the numerical approximation of generalized Korteweg de-Vries (gKdV) equations using \(H^1\)-conforming finite elements and implicit-explicit (IMEX) time stepping. The generalization considered here is one that generalizes the usual hyperbolic flux $\tfrac12 u^2$ to a generic Lipschitz flux $\vecf(u)$. The objective of this work is to construct an approximation technique that is robust (conservative and stable), high order accurate in space and time, agnostic to the choice of hyperbolic flux, and stable up to a Courant–Friedrichs–Lewy (CFL) condition of the type \(\tau = \mathcal{O}(h)\).

\par The KdV equation is a classic partial differential equation in the field of non-linear dispersive waves. It has the special property that 
it admits \textit{soliton} solutions (also known as solitary waves), which are traveling waves that maintain their shape and speed. In general, solutions of the KdV equation can be analytically decomposed into solitons and dispersive waves using the inverse scattering method (\cite{Miura_76}). This phenomenon was first observed numerically by~\cite{Zabusky_65}.

 In applications, the KdV equation is a prototypical equation for shallow water models. It was first derived as a model to explain solitary waves first observed in \cite{Russell_1845}. Since then, it has also been used to study plasma physics (see \cite{Hong_Lee_99}), optics (\cite{Horsley_2016}), and other physical phenomenon (see \cite{Crighton_95}). Its generalizations also have wide applications in plasma physics, traffic flow, and water waves (\cite{Leblond_Mihalache_09, Li_06,Benjamine_Bona_Mahony_97}). A generalization in the context of shallow water models is the Serre-Green-Naghdi equation (\cite{Serre_53, Green_Laws_Naghdi_74}), which is a dispersive system that admits solitons and is used to model shoaling, wave propagation, and dam breaks, for example in \cite{Guermond_Kees_Popov_Tovar22_1}.  

\par Numerically, generalized KdV equations pose several challenges. Explicit time stepping methods require severe CFL restrictions, often of the kind \(\tau = \mathcal{O}(h^3)\), as encountered in \cite[Equation 2.22]{Sanz-Serna82}. This restriction requires a prohibitive number of time steps even when the final time is small; see \cite{Morris_Greig_76}. Implicit time stepping methods may circumvent the CFL restriction issue, for example in \cite{Bona_Dougalis_Karakashian_85, Sanz-Serna_Christie_81, Fu_Shu_2018}, but the nonlinear term leads to a system of nonlinear equations. Lastly, it is desirable for numerical methods to preserve discrete analogs of important properties admitted by the equation (\eg conservation of physical quantities). Particular attention has been paid to the mass and the energy (\(L^2\)-norm) of the solution; see \cite{Karakashain_Ohannes_Xing_Yulong_16, Sanz-Serna82, Winther_80}. 

\par In this paper, we propose a method which uses \(H^1\)-conforming finite elements and an IMEX time-stepping scheme. The method is conservative, high order accurate, and \(\ell^2\) stable at the discrete level. The IMEX split allows us to avoid the restrictive \(\tau = \mathcal{O}(h^3)\) CFL condition by treating the third order term implicitly. We also avoid nonlinear solves by treating the nonlinear flux explicitly, which only requires a CFL condition of the kind \(\tau = \mathcal{O}(h)\). We do not claim originality for the splitting approach; the reader is referred to \cite{Holden_Karlsen_Risebro_Tao_11} for an analysis on operator splitting for the KdV equation. To ensure robustness of the hyperbolic step (\cite{Holden_Lubich_Risebro_11}), we use a method developed by \cite{Guermond_Popov_17} to solve hyperbolic conservation laws. We also reduce the third order derivative to second order by introducing auxiliary variables; doing so permits us to use \(H^1\)-conforming finite elements.

\par The rest of this paper is organized as follows. The model problem is outlined in section \ref{sec:model_problem}. The details of the temporal and space discretizations are provided in section \ref{sec:approximation}. Section \ref{sec: high-order} describes how to make the method high order and conservative, and section \ref{sec:illustrations} contains simulations that demonstrate the theory. The main results of the paper are the \(\ell^2\) stability property (Theorem \ref{thm: imex-stability-theorem}) and the properties of the high-order scheme (Theorem \ref{thm: high-order}). 
\section{The model problem}\label{sec:model_problem}
In this section, we describe the model problem. Let \(d \in \N\) and let \(D = [a,b]^d\subset \R^d\) be the fundamental domain with length \(l_D := b-a\). Let \(\vecf \in \text{Lip}(\R, \R^d)\) be a Lipschitz, vector-valued function, hereafter referred to as the \textit{hyperbolic flux} or simply \textit{flux}. This choice is to indicate that we have in mind functions \(\vecf\) such that the equation
\begin{equation}
    \dt u + \Div(\vecf(u)) = 0, \label{eqn: scalar-conservation-law}
\end{equation}
\noindent is a hyperbolic conservation law. Let \(u_0(x) \in L^\infty(D)\), \(T>0\) and \(\epsilon>0\). We consider the following Cauchy problem for a generalized KdV equation: Find \(u(t,x):[0,T] \times D \to \R\) such that \(u(0,x) = u_0(x)\), \(u\) is periodic with respect to \(D\), and satisfies:

\begin{equation}
    \dt u + \Div(\vecf(u)) + \epsilon \lap(\dx u ) = 0 ~\text{ in \((0,T) \times D\)}. \label{eqn: scalar-dispersive-equation}
\end{equation}
We call \eqref{eqn: scalar-dispersive-equation} a ``generalized'' KdV equation because the hyperbolic flux, \(\vecf(u)\) is generic and does not have to be the usual Burgers flux. Generalizations in the past have included changing the polynomial degree of the flux, \ie \(u^p \dx u\), for example in \cite{Bona_Dougalis_Karakashian_McKinney_95, Karakashian_Ohannes_Makridakis_Charalambos_15}. Other authors have also assumed a generic flux, for example in \cite{Fu_Shu_2018}. 

\par Let us briefly mention several properties of \eqref{eqn: scalar-dispersive-equation} but omit their proofs for brevity. These properties are reflected in the numerical method; see theorem \ref{remark: imex-mass-conservation} for mass conservation and theorem \ref{thm: imex-stability-theorem} for a discrete \(l^2\) stability result. If \(u\) is a smooth solution to \eqref{eqn: scalar-dispersive-equation}, then we have \(\dt \int_D u(x,t) \dif x = 0\) (mass conservation), and if \(\int_D \Div(\vecf(u)) u \dif x =0\), then \( \dt \| u(t, \cdot) \|_{L^2(D)} = 0\) (energy conservation). For example, many generalized KdV equations consist of a polynomial flux, \ie \(f(u) = u^p\) for \(p \in \N\), and such fluxes satisfy these properties. 

\par Concrete examples of \eqref{eqn: scalar-dispersive-equation} include the following: 
\begin{align}
      \dt u + 6 u \dx u + \dxxx u & = 0, \label{eqn: kdv}\quad \text{(KdV)} \\ 
       \dt u + \dx u + \dxxx u & = 0, \quad \text{(Linear-KdV)} \\
        \dt u + 6 u \dx u + \lap( \dx u) & = 0,  \quad \text{(Zakharov–Kuznetsov (ZK) (\cite{Wazwaz_2008}))} \label{eqn: ZK}\\
        \dt u + \dx \left( \frac{1}{p}u^p\right) + \dxxx u & = 0,  \quad \text{(Polynomial gKdV)} \\
        \dt u + \dx u + \dx u^p + \dxxx u & = 0. 
\end{align}
\noindent We pose the model problem in multiple spatial dimensions, but in practice the problems we are interested in are one dimensional.

\section{Low-Order Approximation details}\label{sec:approximation}
In this section, we present the details for the low-order spatial and temporal approximation. It begins with the IMEX time-stepping method while remaining continuous in space. Then, it provides the spatial discretization details for the \(H^1\)-conforming finite elements. Next we provide the fully discrete explicit hyperbolic prediction and implicit dispersive update sub-steps. It ends with a stability result in the discrete \(l^2\)  norm defined in \eqref{eqn: l2-norm} and a mass-conservation result. 

\subsection{Temporal Discretization}
Let \(t^n\) be the current time with \(n \in\set{0:N}\), \(t^0 = 0\) and \(t^N = T\). Let \(\tau^n>0\) be the current time step with \(t^{n+1}:= t^n + \tau^n\). For the remainder of this work, we remove the dependency of \(\tau\) on \(n\) to alleviate the notation, but in practice the time step depends on \(n\). Let \(u^n: D \to \R\) be the approximation at time \(t^n\). We construct \(u^{n+1}\) using an implicit-explicit method, where the hyperbolic flux is treated explicitly and the third order operator is treated implicitly. Hence, \(u^{n+1}\) is found by solving the following equation: 
\begin{equation}
    \frac{u^{n+1}- u^n}{\tau} + \Div(\mathbf{f}(u^n)) + \epsilon \lap (\dx u^{n+1}) = 0. \label{eqn: semi-discrete-approx} 
\end{equation}
We introduce an intermediate stage, \(w^{n+1}\), and solve \eqref{eqn: semi-discrete-approx} by rewriting it as the following system: 
\begin{align}
        w^{n+1} & = u^n - \tau \Div( \mathbf{f}(u^n)), \label{eqn: semi-hyperbolic-prediction} \\
         u^{n+1} + &\tau \epsilon \lap( \dx u^{n+1}) = w^{n+1}.  \label{eqn: semi-disp-update}
\end{align}
Equation \eqref{eqn: semi-hyperbolic-prediction} is an explicit update with the hyperbolic flux, which then becomes the source term in the implicit update \eqref{eqn: semi-disp-update}. 

\subsection{Reformulation}
In this section, we introduce auxiliary variables to reformulate the model and reduce \eqref{eqn: semi-disp-update} to second order in space. Doing so allows us to use \(H^1\)-conforming finite elements, whereas keeping \eqref{eqn: semi-disp-update} would require \(H^2\)-conformity. Since this analysis concerns \eqref{eqn: semi-disp-update}, we consider the following equation: 

\begin{equation}
    \dt u + \epsilon \lap( \dx u ) = 0. \label{eqn: dispersive}
\end{equation}
\noindent The solution satisfies the energy estimate \( \dt \int_D u^2 \dif x =0\). Assuming \(d =2\), our reformulation of \eqref{eqn: dispersive} consists of auxiliary variables \(z,g\) which satisfy the following system: 
\begin{align}
    \dt u + \epsilon \dxx z & + \epsilon \dy \dx g = 0, \label{eqn: disp-reform-u}\\
    z & = \dx u, \label{eqn: disp-reform-z}\\
    g & = \dy u. \label{eqn: disp-reform-g}
\end{align}

\begin{lemma}
    Let \((u,z,g)\) be a smooth periodic solution to \eqref{eqn: disp-reform-u}-\eqref{eqn: disp-reform-g}. Then the following energy conservation equation holds: 
    \begin{equation}
        \dt \int_D u^2 \dif x = 0. \label{eqn: disp-reform-energy}
    \end{equation}
\end{lemma}
\begin{proof}
Multiply equation \eqref{eqn: disp-reform-u} by \(u\), \eqref{eqn: disp-reform-z} by \( \epsilon \dx z\), and \eqref{eqn: disp-reform-g} by \( \epsilon \dx g\). Integrate each over \(D\). Since \(g,z\) are periodic, we deduce \( \int_D z \dx z \dif x = \int_D g \dx g \dif x = 0\). We thus arrive at the following: 
\begin{align*}
        \dt \int_D \frac{u^2}{2} \dif x - \epsilon \int_D \dx z \dx u \dif x & - \epsilon \int_D \dx g \dy u \dif x = 0 \\
        \epsilon\int_D \dx u \dx z \dif x & = 0 \\
        \epsilon \int_D \dy u \dx g \dif x & = 0.
\end{align*}
\noindent Summing these equations and multiplying by \(2\) yields \eqref{eqn: disp-reform-energy}.
\end{proof}

\noindent Motivated by this analysis, we replace \eqref{eqn: dispersive} by the following system: 
\begin{align}
     \frac{u^{n+1} - w^{n+1}}{\tau} + \epsilon \dxx z^{n+1} & + \epsilon \dy \dx g^{n+1} = 0, \label{form: disp-mixed-form-1} \\
    z^{n+1} = & \dx u^{n+1}, \label{form: disp-mixed-form-2}\\
     g^{n+1} = & \dy u^{n+1}. \label{form: disp-mixed-form-3}  
\end{align}

\noindent We note that auxiliary variables have been used before for higher order derivatives. For example, \cite{Yan_Shu_02} reduced the third order problem to a system of first order equation for use in a local discontinuous Galerkin method. 

\subsection{Spatial Discretization}
In this section, we give the spatial discretization and finite element formulation. Recall \(D = [a,b]^d\). Let $\{\calK_h\}_{h>0}$ be a shape-regular sequence of meshes that cover \(D\) exactly. Now let $h>0$ be given, and let \((\Hat{K}, \Hat{P}, \Hat{\Sigma})\) be the reference element. For each $K \in \mathcal{K}_h$, let $T_K: \Hat{K} \to K$ be an affine geometric mapping. For $n \in \N$, let $\widehat{\mathcal{P}}_n$ be the set of all polynomials of degree $n$ on $\Hat{K}$. Furthermore, let $\mathcal{V}:= \set{1:I}$ be an enumeration of the degrees of freedom. Now we consider the space: 
\begin{equation}
    V_h := \set{ v \in C^0(D): v \circ T_K \in \widehat{P} \text{ for all $K \in K_h$ and is periodic with respect to \(D\)}}. \label{set: finite-element-space}
\end{equation}
\noindent Note that \(V_h \subset H^1(D)\). The reference shape functions are denoted by \(\set{\hat{\theta}_i}_{i \in \mathcal{N}}\) with \(\mathcal{N}:= \set{1:n_{sh}}\). These function forms a basis of \(\Hat{P}\) and we assume they satisfy the partition of unity property, i.e., \(\sum_{i \in \mathcal{N}} \hat{\theta}_i(\hat{x}) = 1\) for all \(\hat{x} \in \hat{K}\). We denote the global shape functions by \(\set{\phi_i}_{i \in \mathcal{V}}\). We assume that the global shape functions also have the partition of unity property. For \( i \in \mathcal{V}\), we let \(\mathcal{I}(i):= \set{j \in \mathcal{V}: ~ \phi_j \phi_i  \not\equiv 0 }\). The space approximation of \(u\) at time \(t^n\) is given by \(u_h^n := \sum_{i \in \mathcal{V}} U_i^n \phi_i \in V_h. \) Using the basis functions, we define several quantities: 
\begin{equation}
    m_i := \int_D \phi_i(x) \dif x, \quad \mathbf{c}_{ij}:= \int_D \phi_i(x) \grad \phi_j(x) \dif x, \quad \mathbf{n}_{ij} = \frac{\mathbf{c}_{ij}}{ \|\mathbf{c}_{ij} \|_{l^2}}.
\end{equation}
\noindent We assume that \(m_i >0\) for all \( i \in \mathcal{V}\). The partition of unity property implies that \(\sum_{j \in \mathcal{V}} \mathbf{c}_{ij} = 0\) for all \(i \in \mathcal{V}\), and \( j \in \mathcal{I}(i)\) iff \(i \in \mathcal{I}(j)\). We let \((\cdot, \cdot)\) denote the \(L^2\) inner product and also define the following weighted \(l^2\)-norm on \(V_h\), which we use to prove stability: 
\begin{equation}
    \| u_h \|_{l^2}^2 := \sum_{i \in \mathcal{V}} U_i^2 m_i. \label{eqn: l2-norm}
\end{equation}
\subsubsection{Explicit Hyperbolic Prediction}\label{sec: low-order-explicit} In this section, we outline the hyperbolic sub-step by fully discretizing \eqref{eqn: semi-hyperbolic-prediction} using a graph-viscosity technique developed in \citep{Guermond_Popov_17}. Given \(u_h^n \in V_h\), compute \(w_h^{n+1} \in V_h\) by:
\begin{equation}
    m_i \frac{W^{n+1}_i - U_i^n}{\tau} + \sum_{j \in \mathcal{I}(i)}\mathbf{f}(U_j^n)\cdot \mathbf{c}_{ij} - d_{ij}^n(U_j^n - U_i^n) = 0. \label{form: explicit-hyperbolic-prediction}
\end{equation}
\noindent Here, \(d_{ij}^n := \max( \lambda_{\max}(\mathbf{n}_{ij}, U_i^n, U_j^n) \| \mathbf{c}_{ij} \|_{l^2}, \lambda_{\max}(\mathbf{n}_{ji}, U_j^n, U_i^n) \| \mathbf{c}_{ji}\|_{l^2})\) is called graph viscosity, where \(\lambda_{\max}( \mathbf{n}_{ij}, U_i^n, U_j^n)\) is a bound on the max wave speed on the one-dimensional Riemann problem
\begin{equation}
    \dt u + \Div ( \mathbf{f}(u) \cdot \mathbf{n}_{ij}) = 0; \quad u_0(x) = \begin{cases}
        U_i^n & x < 0 \\ U_j^n & x>0
    \end{cases}; \quad \mathbf{n}_{ij}:= \mathbf{c}_{ij}/ \| \mathbf{c}_{ij} \|_{l^2}.
\end{equation}
\par Using the graph viscosity, \eqref{form: explicit-hyperbolic-prediction} has several key properties, which we recall in the following lemma. 

\begin{lemma}\label{lemma: maximum-principle}
    Let \( n \in \N\) and \(u_h^n \in V_h\). Assume \(\widehat{\theta}_i(x) \geq 0\) for all \( x \in \widehat{K}\) and \(i \in \mathcal{V}\). Assume the following CFL conditions holds: 
\begin{equation}
    \min_{i \in \mathcal{V}} \left( 1 - 2\tau\frac{\sum_{j \in \mathcal{I}(i)\setminus\{i\}}d_{ij}^n}{m_i} \right) \geq 0. \label{eqn: idp-cfl-condition}
\end{equation}

\noindent Then the following statements hold true. 
\begin{enumerate}
    \item Let \(U_i^{m,n}:= \min_{j \in \mathcal{I}(i)}\{U_j^n\}, ~ U_i^{M,n}:= \max_{j \in \mathcal{I}(i)} \{U_j^n\}\). Then, \\ \(W_i^{n+1} \in  [U_i^{m,n}, U_i^{M,n}]\).
    \item Let \(U^n_{\min} := \min_{j \in \mathcal{V}} \{U_j^n\}\) and \(U^n_{\max} := \max_{j \in \mathcal{V}} \{U_j^n\}\). Then \(w_h^{n+1}(x) \in\) \\ \( [U_{\min}^n, U_{\max}^n] ~ \forall x \in D.\) 
    \item Let \( \eta \in C^2(D)\) be convex. Define \(\vecq(u): \R \to \R^d\) satisfying \(q_l'(u) = \eta'(u) f'_l(u)\), \ie \((\eta, \vecq)\) form an entropy pair for \eqref{eqn: scalar-conservation-law}. Then the following discrete entropy inequality holds: 
    \begin{align}
         \frac{m_i}{\tau} \left( \eta(W_i^{n+1}) - \eta(U_i^n) \right) + \int_D \Div\left(\sum_{j \in \mathcal{I}(i)} \vecq(U_j^n) \phi_j
         \right) \phi_i \dif x \label{eqn: discrete-entropy}\\ - \sum_{j \in \mathcal{I}(i)} d_{ij}^n( \eta(U_j^n)- \eta(U_i^n)) \leq 0 \nonumber. 
    \end{align}
\end{enumerate}
\end{lemma}

\begin{proof}
    See \cite{Guermond_Ern_FE_III}.
\end{proof}

\noindent Now we prove an \(\ell^2\) stability result for the hyperbolic prediction. 

\begin{corollary}[Hyperbolic stability]\label{cor: l2-stability}
    Let \(\eta(u) = \frac{u^2}{2}\). Assume the conditions of lemma \ref{lemma: maximum-principle} are met. Then, the following \(\ell^2\) bound holds: 
    \begin{equation}
    \| w_h^{n+1} \|_{\ell^2} \leq  \| u_h^n \|_{\ell^2}. \label{eqn: hyperbolic-l2-stability}
\end{equation}

\end{corollary}
\begin{proof}
    Sum \eqref{eqn: discrete-entropy} over \(i \in \mathcal{V}\). The partition of unity property and periodic boundary conditions imply the second term vanishes; the third term is antisymmetric, and so vanishes under the sum. By invoking the definition of \(\eta\) the proof is complete. 
\end{proof}

\begin{lemma}[Hyperbolic mass conservation]
\label{lemma: hyperbolic-mass-conservation}
Let \(u_h^n \in V_h\) and let \(w_h^{n+1}\) be computed via \eqref{form: explicit-hyperbolic-prediction}. Then, \(\int_D w_h^{n+1} \dif x = \int_D u_h^n \dif x\). 
\end{lemma}
\begin{proof}
    We refer the reader to remark 81.4 in \citep{Guermond_Ern_FE_III}.
\end{proof}

\subsubsection{Implicit Dispersive Update}\label{sec: low-order-disp-update}

In this section, we outline the spatial discretization for the dispersive sub-step \eqref{form: disp-mixed-form-1}-\eqref{form: disp-mixed-form-3}. Let \(w_h^{n+1} \in V_h\). The fully discrete dispersive update is given by: Find 
\((u_h^{n+1}, z_h^{n+1}, g_h^{n+1}) \in V_h \times V_h \times V_h\) so that, for all \((v_h, r_h, s_h) \in V_h \times V_h \times V_h\): 
\begin{align}
      \left(\frac{u_h^{n+1} - w_h^{n+1}}{\tau},v_h \right) - \epsilon( \dx z_h^{n+1}, \dx v_h) - & \epsilon (\dx g_h^{n+1}, \dy v_h) - \epsilon c(z_h^{n+1} - \dx u_h^{n+1}, \dx v_h) -  \nonumber \\ \epsilon c & (g_h^{n+1} - \dy u_h^{n+1}, \dy v_h) = 0, \label{form: implicit-discrete-1} \\
       ( z_h^{n+1}- \dx u_h^{n+1}, & c r_h - \dx r_h) = 0, \label{form: implicit-discrete-2}  \\
             (g_h^{n+1} - \dy u_h^{n+1}, & c s_h - \dx s_h) =0. \label{form: implicit-discrete-3}
\end{align}
\noindent where \(c>0\) is a dimensional constant that scales by the inverse of the length scale.  

\begin{lemma}[Dispersive stability]
\label{lemma: implicit-stability}
    Let \(w_h^{n+1} \in V_h\) and let \(u_h^{n+1}\in V_h\) be computed by solving \eqref{form: implicit-discrete-1}-\eqref{form: implicit-discrete-3}. Then, for all \(\tau >0\), the following holds:
    \begin{multline}
         \| u_h^{n+1} \|_\Ltwo^2 + \| u_h^{n+1} - w_h^{n+1} \|_\Ltwo^2 + 2\tau \epsilon c \| z_h^{n+1} - \dx u_h^{n+1} \|_\Ltwo^2 + \\ 2\tau \epsilon c \| g_h^{n+1} - \dy u_h^{n+1} \|_\Ltwo^2 = \| w_h^{n+1} \|_\Ltwo^2. \label{eqn: implicit-stability}
    \end{multline}
\end{lemma}

\begin{proof}

We test \eqref{form: implicit-discrete-1} with \(v_h = u_h^{n+1}\), \eqref{form: implicit-discrete-2} with \(r_h = z_h^{n+1}\), and \eqref{form: implicit-discrete-3} with \(s_h = g_h^{n+1}\), which yields: 
\begin{align}
     \frac{\| u_h^{n+1}\|_{L^2}^2}{2} + \frac{\| u_h^{n+1} - w_h^{n+1} \|_{L^2}}{2} & - \frac{\| w_h^{n+1}\|_{L^2}^2}{2} - \tau \epsilon ( \dx z_h^{n+1}, \dx u_h^{n+1} ) \label{implicit-stability-1} \\
     - \tau \epsilon (\dx g_h^{n+1}, \dy u_h^{n+1}) - & \tau \epsilon c (z_h^{n+1} - \dx u_h^{n+1}, \dx u_h^{n+1} ) \nonumber\\  
     - & \tau \epsilon c( g_h^{n+1} - \dy u_h^{n+1}, \dy u_h^{n+1} ) = 0, \label{implicit-stability-eq-2} \nonumber \\
              c (z_h^{n+1}  - \dx u_h^{n+1},  z_h^{n+1}) -& (- \dx u_h^{n+1}, \dx z_h^{n+1}) = 0, \\
          c ( w_h^{n+1} - \dy u_h^{n+1},  w_h^{n+1}) &- ( - \dy u_h^{n+1}, \dx w_h^{n+1}) = 0 \label{implicit-stability-eq-3}.   
\end{align}

\noindent Multiplying equation \eqref{implicit-stability-eq-2} and \eqref{implicit-stability-eq-3} equations by \(\tau\epsilon\) and adding them to \eqref{implicit-stability-1} then gives 
\begin{align}
      \frac{\| u_h^{n+1}\|_{L^2}^2}{2} + \frac{\| u_h^{n+1} - w_h^{n+1} \|_{L^2}}{2} - \frac{\| w_h^{n+1}\|_{L^2}^2}{2} + \nonumber \\ \tau \epsilon c \| z_h^{n+1} - \dx u_h^{n+1} \|_{L^2}^2 + \tau c \epsilon \| g_h^{n+1} - \dy u_h^{n+1} \|_{L^2}^2 =0.
\end{align}
\end{proof}

\begin{remark}[Well-posedness of the dispersive update]\label{remark: discrete-disp-well-posed}
We conclude the dispersive update problem \eqref{form: implicit-discrete-1}-\eqref{form: implicit-discrete-3} is well posed by using the stability result and invoking the Rank-Nullity theorem. 
\end{remark}

\begin{remark}[Lumped mass for implicit dispersive update]
\label{remark: implicit-lumped-stability}
In order to prove a stability result for the IMEX method, we need to approximate the time derivative in \eqref{form: disp-mixed-form-1} using the lumped mass. The problem becomes: for all \( i \in \mathcal{V}\), find \(U_i^{n+1}\) so that: 
\begin{align}
      m_i \frac{U_i^{n+1} - W_i^{n+1}}{\tau} - \epsilon( \dx z_h^{n+1}, \dx \phi_i) - \epsilon(\dx g_h^{n+1}, \dy \phi_i) - \label{form: implicit-lumped-mass} \\ \epsilon c(z_h^{n+1} - \dx u_h^{n+1}, \dx \phi_i) - \epsilon c (g_h^{n+1} - \dy u_h^{n+1}, \dy \phi_i) = 0, \nonumber 
\end{align}
\noindent while \eqref{form: implicit-discrete-2}-\eqref{form: implicit-discrete-3} remain the same. The stability result is similar to lemma \ref{lemma: implicit-stability}, but with different norms: 
\begin{align}
       \| u_h^{n+1} \|_{l^2}^2 + \| u_h^{n+1} -w_h^{n+1} \|_{l^2}^2 + 2\epsilon \tau \| z_h^{n+1} - \dx u_h^{n+1} \|_{L^2}^2 + \\ 2 \epsilon\tau c \| g_h^{n+1} - \dy u_h^{n+1} \|_{L^2}^2 = \| w_h^{n+1} \|_{l^2}^2. \nonumber
\end{align}
\noindent The proof proceeds as follows: multiply \eqref{form: implicit-lumped-mass} by \(U_i^{n+1}\) and sum over \(i \in \mathcal{V}\). As the restrictions are the same, the proof continues as in lemma \ref{lemma: implicit-stability}.  
\end{remark}

\begin{remark}[Mass Conservation]\label{lemma: implicit-mass-conservation}
   By testing \eqref{form: implicit-discrete-1} with \(v_h = 1\), we deduce that dispersive update is mass conservative. The result holds as well if we use the lumped mass. 
\end{remark}

\subsubsection{Euler-IMEX}\label{sec: Euler-IMEX}
This section contains the stability and conservation results for the complete IMEX method. These results are direct consequences of the lemmas regarding the explicit and implicit sub-problems. 

\begin{theorem}[Stability with lumped mass]
\label{thm: imex-stability-theorem}
    Let \( n \in \N\). Let \(u_h^n \in V_h\), and let \(w_h^{n+1}\) be computed using the explicit hyperbolic prediction \eqref{form: explicit-hyperbolic-prediction}. Let \(u_h^{n+1} \in V_h\) be computed via the implicit dispersive update with lumped mass matrix, \ie solving \eqref{form: implicit-lumped-mass} and \eqref{form: implicit-discrete-2}-\eqref{form: implicit-discrete-3}. Then, for \(\tau \leq \tau^*\) with 
    \begin{equation*}
        \tau^*:= \min_{i \in \mathcal{V}} \frac{m_i}{ 4\sum_{j \in \mathcal{I}(i)} d_{ij}^n}, \label{cond: low-order-CFL}
    \end{equation*}
\noindent the following inequality holds: 
\begin{equation}
    \| u_h^{n+1} \|_{l^2} \leq \|u_h^n \|_{l^2}. \label{eqn: imex-stability-equation}
\end{equation}
\end{theorem}
\begin{proof}
    Let \(u_h^n \in V_h\) and \(\tau \leq \tau^\ast\). Then, by corollary \ref{cor: l2-stability}, since \( \tau \leq \tau^\ast\), \( \| w_h^{n+1} \|_{l^2} \leq \| u_h^{n} \|_{l^2}\). The analysis in remark \ref{remark: implicit-lumped-stability} implies that \(\| u_h^{n+1} \|_{l^2} \leq \| w_h^{n+1} \|_{l^2}\). Thus, \( \| u_h^{n+1} \|_{l^2} \\ \leq  \| u_h^n \|_{l^2} \). 
\end{proof}

This theorem shows the benefit of utilizing an IMEX method: since \(m_i = \mathcal{O}(h)\), the method only requires a CFL condition of \(\tau = \mathcal{O}(h)\) for stability. Furthermore, we require no nonlinear solves, as the nonlinearity is computed explicitly. 

\begin{remark}[Mass Conservation for Euler-IMEX]\label{remark: imex-mass-conservation}
If \(u_h^0 \in V_h\), then \( \int_D u_h^{n+1} \dif x \\ = \dots = \int_D u_h^0 \dif x\). This fact follows from lemma \ref{lemma: hyperbolic-mass-conservation} and the analysis done in remark \ref{lemma: implicit-mass-conservation}. 
\end{remark}

\section{High Order}\label{sec: high-order}
This section contains the details to make the IMEX method high order in time and space. The idea is to make the hyperbolic prediction \eqref{form: explicit-hyperbolic-prediction} high order in space by using higher degree polynomials in conjunction with high order graph viscosity and a limiting operator. The dispersive update \eqref{form: disp-mixed-form-1}-\eqref{form: disp-mixed-form-3} is made high order in space by using higher order polynomials in the definition of the finite element space. To make the time integration high order, we follow the work done in \cite{Guermond_Ern_23} and use a particular set of explicit and diagonally implicit Runge-Kutta methods (ERK and DIRK methods, respectively). These methods are constructed to be high order in time and mass conservative.

\subsection{High Order Hyperbolic Prediction}\label{sec: high-order-hyperbolic-prediction}
Let us outline the procedure to make the hyperbolic prediction high order in space. Begin by constructing a high order graph viscosity. For \(n \in \N\), let \(\{ \psi_k^n \}_{k \in \mathcal{V}}\) be a sequence of positive real numbers. The high order graph viscosity is given by: 
\begin{equation}
    d_{ij}^{H,n} := d_{ij}^n \max(\psi_i^n, \psi_j^n). \label{eqn: high-order-graph-viscosity}
\end{equation}
\noindent There are several ways to choose \(\psi_k^n\). Examples include \textit{smoothness based graph viscosity}, where \(\psi_k^n\) are chosen based upon the local smoothness of the solution, \textit{greedy graph viscosity}, where \(\psi_k^n\) are based on a set of values derived from the solution and the \(d_{ij}^n\)s, and \textit{entropy-based graph viscosity}, where the \(\psi_k^n\) are chosen using entropy inequalities. Further details on all of these examples can be found in \citep{Guermond_Ern_FE_III}.

\par The next step is to utilize the consistent mass matrix, given by \(m_{ij} := \int_D \phi_i \phi_j \dif x\), in place of the lump-mass matrix. These two changes give rise to the following hyperbolic prediction: 
\begin{equation}
    \sum_{j \in \mathcal{I}(i)} \frac{m_{ij}}{\tau}(W_j^{n+1} - U_j^n) + \sum_{j \in \mathcal{I}(i)} \left( \mathbf{f}(U_j^n) \cdot \mathbf{c}_{ij} - d_{ij}^{H,n}(U_j^n  - U_i^n) \right) = 0. \label{form: high-order-explicit}
\end{equation}
With this expression and \(\eqref{form: explicit-hyperbolic-prediction}\) in mind, we define the following maps: \(\mathbf{F}^L: \R^I \to \R^I\) and \(\mathbf{F}^H: \R^I \to \R^I\) via: 
\begin{align}
        (\mathbf{F}^L(U))_i & = \sum_{j \in \mathcal{I}(i)} \vecf(U_j) \cdot \mathbf{c}_{ij} - d_{ij}^n(U_j - U_i), \label{form: low-order-flux}  \\
            (\mathbf{F}^H(U))_i & = \sum_{j \in \mathcal{I}(i)} \vecf(U_j)\cdot \mathbf{c}_{ij} - d^{H,n}_{ij}(U_j - U_i). \label{form: high-order-flux}
\end{align}
The low and high order hyperbolic predictions can be written using \eqref{form: low-order-flux} and \eqref{form: high-order-flux} as follows: 
\begin{align}
      \mass^L W^{L,n+1}&  = \mass^L U^n - \tau \mathbf{F}^L(U^n), \label{eqn: low-order-hyperbolic-prediction}\\
      \mass^H W^{H,n+1} & = \mass^H U^n - \tau \mathbf{F}^H(U^n), \label{eqn: high-order-hyperbolic-prediction}
\end{align}
\noindent where \((\mass^L)_i:= \int_D \phi_i \dif x\) denotes the lumped mass matrix and \((\mass^H)_{ij} = m_{ij} \) denotes the consistent mass matrix.  
\par The next step consists of using a limiting operator. Since there are many such operators in the literature, we formalize this technology via the following assumption. 

\begin{assumption}\label{assumption: limiter-assumption}
    There exists a map, called a \textit{limiter}, \(\ell:[a,b]^I \times \R^I \times \R^I \to \R^I\) such that the following properties hold: 
    \begin{enumerate}
        \item (Maximum principle) If \(\mass^L U^n - \tau \mathbf{F}^L(U^n) \in [a,b]^I\), then \\\(\ell(U^n, \mathbf{F}^L(U^n), \mathbf{F}^H(U^n))\in [a,b]^I\). 
        \item (Mass conservation) Let \(U^{n+1} := \ell( U^n, \mathbf{F}^L(U^n), \mathbf{F}^H(U^n))\). Then \\ \(\sum_{i \in \stencil} U_i^{n+1} m_i = \sum_{i \in \stencil} U_i^n m_i\). 
    \end{enumerate}
\end{assumption}

Examples of such a limiter include \textit{flux corrected transport}, found in \cite{Zalesak_79} and \cite{Kuzmin_12} and \textit{convex limiting}, which relies upon quasi-concave functionals, found in \citep{Guermond_Ern_FE_III}.    

\par We now have all the pieces to describe the high order (in space) hyperbolic prediction: given \(u_h^n \in V_h\), compute \(w_h^{L,n+1}\) and \(w_h^{H,n+1}\) via \eqref{eqn: low-order-hyperbolic-prediction} and \eqref{eqn: high-order-hyperbolic-prediction}. Then we compute \(w_h^{n+1}\) via: 
\begin{equation}
    W^{n+1} = \ell(U^{n}, \mathbf{F}^L(U^n), \mathbf{F}^H(U^n)). \label{eqn: hyperbolic-limiter} 
\end{equation}
\begin{lemma}
    Let \(n \in \N\) and let \eqref{cond: low-order-CFL} hold. Let \(u_h^n \in V_h\) and let \(w_h^{L,n+1}\) be computed using \eqref{eqn: low-order-hyperbolic-prediction}, \(w_h^{H,n+1} \in V_h\) with \eqref{eqn: high-order-hyperbolic-prediction}, and \(w_h^{n+1} \in V_h\) by applying \eqref{eqn: hyperbolic-limiter} to \(w_h^{L,n+1}\) and \(w_h^{H,n+1}\). Then, \( \int_D w_h^{n+1} \dif x = \int_D u_h^n \dif x \). \label{lemma: high-order-explicit-conservation}
\end{lemma}
\begin{proof}
     By lemma \ref{lemma: hyperbolic-mass-conservation} \(\int_D w_h^{L,n+1} \dif x = \int_D u_h^n \dif x \). Then, since the limiter is conservative by assumption, we obtain \(\int_D w_h^{n+1} \dif x = \int_D w_h^{L,n+1} \dif x \). Combining these two equalities gives the result. 
\end{proof}

\subsection{High Order Dispersive Update}\label{sec: high-order-disp-update}
\par To alleviate notation, we introduce an operator \(\mathbf{G}:V_h \to V_h\) defined by 
\begin{align}
        (\mathbf{G}(u_h), v_h) = - \epsilon( \dx z_h(u_h), \dx v_h) - \epsilon( \dx q_h(u_h), \dy v_h) - \\ \epsilon c( z_h(u_h) - \dx u_h, \dx v_h) - \epsilon c( q_h(u_h) - \dx u_h, \dx v_h), \qquad \forall v_h \in V_h \nonumber \label{form: discrete-disp-operator}
\end{align}
\noindent where \(z_h(u_h), q_h(u_h)\) are defined by satisfying the following system: 
\begin{align}
    (z_h - \dx u_h, c s_h - \dx s_h) = 0 & ~~~ \forall s_h \in V_h,  \label{eqn: disp-restriction-1} \\
    (q_h - \dy u_h, c r_h - \dx r_h) = 0 & ~~~ \forall r_h \in V_h.  \label{eqn: disp-restriction-2} 
\end{align}
\begin{proposition}
    The map \(\mathbf{G}\) is well defined. 
\end{proposition}
\begin{proof}
    Let \(u_h \in V_h\). It suffices to show that system \eqref{eqn: disp-restriction-1}-\eqref{eqn: disp-restriction-2} can be solved when \(u_h\) is known. Suppose \(u_h = 0\). Then \((z_h, cs_h - \dx s_h) = 0\). Now we test with \(s_h = z_h\). This yields \(c\| z_h\|_{L^2}^2 + ( z_h , \dx z_h) = 0\). However, since \(z_h\) is periodic, this leaves us with \( c \| z_h \|_{L^2}^2  = 0\), and thus \(z_h = 0\). The same argument applies to \(q_h\). Thus the system is well-posed when \(u_h\) is given, and thus \(\mathbf{G}\) is well-defined.
\end{proof}

\begin{remark}[Lack of limiting on dispersive update]
Contrary to \citep{Guermond_Ern_23}, we do not limit the implicit update. There is no limiting because the dispersive step does not satisfy a maximum principle.  
\end{remark}

\subsection{High order in time}\label{sec: high-order-imex}
This section contains the method for solving \eqref{eqn: scalar-dispersive-equation} with high order accuracy in time; we follow the work of \citep{Guermond_Ern_23}, with some modifications. It consists of an explicit Runge-Kutta (ERK) method with a diagonally implicit Runge-Kutta method that has a fully explicit first stage (EDIRK), and we assume that both schemes consist of \(s\) stages with \( s \geq 2\).  

\begin{table}[H]
    \centering
    \footnotesize
    \begin{tabular}{cc}
       \begin{tabular}{c|ccccc}
        0 & 0  \\
        $c_2$ & $a_{2,1}^e$  \\
        $c_3$ & $a_{3,1}^e$ & $0$ \\
        $\vdots$ & $\vdots$ & $\ddots$ & $\ddots$ \\
        $c_s$ & $a_{s,1}^e$ & $\dots$ & $a_{s,s-1}^e$ & $0$ \\ \hline
        & $a_{s+1,1}^e$ &  \dots & $a_{s+1,s-1}^e$ & $a^e_{s+1, s}$
    \end{tabular}   & \begin{tabular}{c|cccccc}
        0 & 0  \\
        $c_2$ & $a_{2,1}^i$ & $a_{2,2}^i$ \\
        $c_3$ & $a_{3,1}^i$ & $a_{3,2}^i$ \\
        $\vdots$ & $\vdots$ & $\ddots$ & $\ddots$ \\
        $c_s$ & $a_{s,1}^i$ & $a_{s,2}^i$ & $\dots$ & $a_{s,s}^i$ \\ \hline
        & $a_{s+1,1}^i$ & $a_{s+1,2}^i$ & \dots &$a_{s+1,s}^i$ & 0
    \end{tabular}  \\
    \end{tabular}
    \caption{Left: Explicit Runge-Kutta (ERK) tableau. Right: Diagonally Implicit Runge-Kutta (DIRK) tableau.}
    \label{tab:butcher-tableaux}
\end{table}
\noindent We assume that \((c_l)_{l \in \set{1:s}}\) is nondecreasing and define the following quantity: 
\begin{equation}
    \Delta c^{\text{max}} := \max_{2 \leq l \leq s+1} (c_{l} - c_{l-1}). 
\end{equation}

\noindent Observe that \(\Delta c^{\text{max}} \geq \frac{1}{s}\) and \(\Delta c^{\text{max}} = \frac{1}{s}\) when all stages are equi-distributed, \ie \(c_l = \frac{l-1}{s}, ~l \in \{1:s+1\}\). We further assume 
\begin{equation}
\sum_{l \in \set{1:j}} a_{j,l}^e = \sum_{l \in \set{1:j}} a_{j,l}^e = c_j \quad \forall j \in \set{1:s}. 
\end{equation}
We are now prepared to explain the high order IMEX method. Given \(U^n \in \R^I\), set \(U^{n,1} := U^n\), and decompose each stage \(l \in \set{2:s+1}\) into the following steps: 

\begin{equation}
    U^{n,l-1} \underbrace{ \longrightarrow (W^{n,L,l}, W^{n,H,l}) \longrightarrow W^{n,l}}_{\text{Hyperbolic prediction}}
    \underbrace{\longrightarrow U^{n,l}}_{\text{Dispersive Update}}.
\end{equation}

\noindent The update is given by \(U^{n+1} = U^{n,s+1}\). 

\subsubsection{Hyperbolic Prediction}
We first compute the low-order and high-order hyperbolic predictions defined by: 
\begin{align}
        \mass^L W^{L,l} =&\mass^L U^{n,l-1} + \tau (c_l - c_{l-1}) \mathbf{F}^L(U^{n,l-1}), \label{form: imex-hyperbolic-low-order}\\
            \mass^H W^{H,l} =&  \mass^HU^{n, l-1}  + \tau \sum_{k \in \set{1:l-1}} (a^e_{l,k} - a_{l-1,k}^e) \mathbf{F}^H(U^{n,k}), \label{form: imex-hyperbolic-high-order} \\
                W^{n,l} = & \ell(U^{n,l-1}, \mathbf{F}^L(U^{n, l-1}), \mathbf{F}^H(U^{nl-1})). \label{form: imex-hyperbolic-limited-state}
\end{align}
\vspace{-10pt}
\begin{remark}[Incremental method]
    The stages are presented in incremental form for two reasons. The first is that the limiter requires that the low and high order updates come from the same stage, namely \(U^{n,l-1}\). Secondly, the time step size is maximized so that each stage uses the \(\tau^\ast\) time step so that a single IMEX step advances \(\tau \times s\) in time. Further commentary is provided in remark \ref{remark: max-efficient-implementation}. 
\end{remark}
\subsubsection{Dispersive Update}
 Once we compute \(W^{n,l}\), we compute the parabolic update at stage \(l \in \{2:s+1\}\) by solving the following system: 
\begin{equation}
    \mass^H U^{n,l} + \tau \mathbf{G}(U^{n,l}) = \mass^H W^{n,l} - \tau \sum_{k \in \set{1: l-1}} (a_{l,k}^i - a_{l-1,k}^i) \mathbf{G}(U^{n,k}). \label{form: imex-disp-upate}
\end{equation}
\begin{theorem}\label{thm: high-order}
Let \( n \in \N\). Let \(\tau^\ast\) be same as in \ref{lemma: maximum-principle} and let
\begin{equation}
    \tau \in \left( 0, \frac{\tau^\ast}{\Delta c^{\max}} \right]. \label{cond: imex-CFL-condition}
\end{equation}
\noindent Let \(U^n \in \R^I\). Consider the \(s\)-stage IMEX scheme composed of the \eqref{form: imex-hyperbolic-low-order} - \eqref{form: imex-disp-upate} for \(l \in \set{2:s+1}\). Then the scheme satisfies the following: 
\begin{enumerate}
    \item It is well defined.
    \item It is conservative.
\end{enumerate}
\end{theorem}
\begin{proof}
    Let \(U^n \in \R^I\) and assume \ref{cond: imex-CFL-condition} holds true. We are going to show that \(U^{n,l}\) is well defined and \(\sum_{i \in \mathcal{V}} U^{n,l}_i m_i = \sum_{i \in \mathcal{V}} U_i^{n}m_i\). We do so by induction. Since \(U^n = U^{n,1}\), conditions (1) and (2) are trivially satisfied. Now let \(U^{n,l-1} \in \R^I\). Since \eqref{form: imex-hyperbolic-low-order} -\eqref{form: imex-hyperbolic-limited-state} are either explicit or utilize the limiter, \(W^{n,l}\) is well-defined. Furthermore, lemma \ref{lemma: high-order-explicit-conservation} implies that \(\sum_{i \in \mathcal{V}}W_i^{n,l}m_i = \sum_{i \in \mathcal{V}} U_i^{n,l-1} m_i\). Finally, remark \ref{remark: discrete-disp-well-posed} implies that \eqref{form: imex-disp-upate} can always be solved for \(U^{n,l}\); the fact that \(( \mathbf{G}(U), \mathbf{1}) = 0\) for all \(U \in \R^I\) and the analysis done in lemma \ref{lemma: implicit-mass-conservation} imply that \(\sum_{i \in \mathcal{V}} U_i^{n,l} = \sum_{i \in \mathcal{V}} W_i^{n,l} m_i = \sum_{i \in \mathcal{V}} U_i^{n,l-1} m_i\). This completes the proof. 
\end{proof}

\begin{remark}[Maximally Efficient Implementation]
In the case when the \(c_j\) coefficients are equi-distributed, we modify the implementation so that the computed time step \(\tau^\ast>0\) from the low-order hyperbolic prediction may be used at each stage. To do so, note \(\Delta c^{\text{max}} = c_l - c_{l-1} = \frac{1}{s}\) for all \(j \in \set{2:s+1}\). Thus, setting \(\tau = \frac{\tau^\ast}{\Delta c^{\text{max}}}\) implies \(\tau = \tau^\ast \times s\). Looking at equations \eqref{form: high-order-explicit} and \eqref{form: imex-disp-upate}, we can apply \(s\) to the coefficients, and so when assembling the sum, we use \(s(a_{l,k} - a_{l-1,k})\). This allows us to use \(\tau^\ast\) as the time step. \label{remark: max-efficient-implementation}
\end{remark}

\subsection{Extensions}
Let us mention a few possible extensions of this method. The first involves including a viscous term \(- \nu \dxx u\) in the model, which is sometimes known as the KdV-Burgers Equation (\cite{Canosa_Gazdag_77}). The viscous term is treated implicitly. Given the method's agnosticism to the hyperbolic flux, it can be extended to equations of the form: 
\begin{equation}
    \dt u + \dx (f(u)) - \nu \dxx u + \alpha \dxxx u = 0. \label{eqn: kdv-burgers}
\end{equation}
The method can also be extended to systems of KdV-like equations. One such case is found in \cite{Bona_Chen_Karakashian_Wise_18}, where the authors consider systems of the form: 
\begin{align}
    \dt u + \dxxx u + \dx P(u,v) & = 0, \\
    \dt v + \dxxx v + \dx Q(u,v) & = 0.
\end{align}

\noindent The functions \(P\) and \(Q\) are assumed to produce a hyperbolic flux \(\mathbf{L}(\mathbf{U}) := \\ (P(u,v), Q(u,v))^T\), where \(\mathbf{U}:= (u,v)^T\). Then the low order method consists of solving
\vspace{-10pt}
\begin{equation}
    m_i \frac{\mathbf{W}_i^{n+1} - \mathbf{U}_i}{\tau} + \sum_{j \in \mathcal{I}(i)} \mathbf{L}(\mathbf{U}_j^n) \cdot \mathbf{c}_{ij} - d_{ij}^n(\mathbf{U}_j^n - \mathbf{U}_i^n) = 0,
\end{equation}

\noindent followed by solving the dispersive parts for \(u\) and \(v\) separately. 

\section{Numerical simulations}\label{sec:illustrations}

To illustrate the method's capabilities, we present several numerical simulations (for both verification and validation). The results include convergence tables which demonstrate the method's accuracy and benchmark tests to show the method's ability to solve standard problems. The numerical method is implemented in a stand-alone code based on the \verb|deal.II| finite element library (\cite{dealII95}) and the high-performance code \verb|ryujin| (\cite{ryujin-2021-1, ryujin-2021-3}). 

\subsection{Verification}
In this section, we present convergence tables to demonstrate the method's accuracy. We compute convergence rates by taking a known analytic solution and measuring the error between the exact solution and the numerical solution. For the verification test, we solve the KdV equation \eqref{eqn: kdv} on the domain \(D = (-10,10)\) with exact solution 
\begin{equation}
    u(t,x) = 2 \sech^2( x - 4t), \label{eqn: kdv-solution}
\end{equation}
and final time \(T>0\). At \(t =T\), we measure the error between the exact and numerical solutions using the quantity \(\text{err}_\infty(T) := \| u_h^N - u(T) \|_{L^\infty(D)}/ \| u(T) \|_{L^\infty(D)}\). 

\par We run the test using the Euler-IMEX, IMEX(2,2;1), IMEX(3,3;1), and IMEX(5,4 ;1) methods found in \citep{Guermond_Ern_23}. For the high-order graph viscosity, we set \(d_{ij}^{H,n} =0\) and use the convex limiting technique detailed in \citep[Chapter 83.2.4]{Guermond_Ern_FE_III}; also see \cite{Guermond_Nazarov_Popov_Tomas_18}. We report the results in Table \ref{tab: imex-convergence}. 

\subsubsection{Convergence} Let us briefly describe Table \ref{tab: imex-convergence}. There are four subtables, each reporting the results for a different test. Each subtable reports the number of degrees of freedom ``\# Dofs'' used for the \(u\)-variable (that is, the degrees of freedom used for the auxiliary variable are not accounted for), the error \(\text{err}_\infty(T)\), and the convergence rate. The top left subtable shows the results using the low-order Euler-IMEX method using \(\mathbb{P}_1\) elements on seven meshes with final time \(T =1 \). The top right substable shows the results for the second order IMEX(2,2;1) method with \(\mathbb{P}_1\) elements on seven meshes with final time \(T=0.5\). The bottom left subtable shows the results for the IMEX(3,3;1) method with \(\mathbb{P}_2\) elements on six meshes with final time \(T =0.5\), and the bottom right shows the results for the IMEX(5,4;1) method with \(\mathbb{P}_3\) elements on six meshes with final time \(T = 0.5\). 
\begin{table}[h!]
    \centering
    \begin{tabular}{ccc|ccc}
    & Euler-IMEX & & & IMEX(2,2;1) & \\
    \#Dofs  & \(\text{err}_\infty(1)\) & rate & \#Dofs & \(\text{err}_\infty(0.5)\) & rate \\ \hline
    129 & 4.32 E-01 & -- & 33 & 3.40 E-01 & -- \\
    257 & 2.95 E-01 & 0.55 & 65 & 1.15 E-01 & 1.60 \\
    513 & 1.79 E-01 & 0.72 & 129 & 3.15 E-02 & 1.89 \\
    1025 & 1.00 E-01 & 0.84 & 257 & 7.82 E-03 & 2.02 \\
    2049 & 5.31 E-02 & 0.91 & 513 & 1.97 E-03 & 1.99 \\
    4079 & 2.73 E-02 & 0.96 & 1025 & 4.90 E-04 & 2.01 \\ 
    8193 & 1.39 E-02 & 0.97 & 2049 & 1.22 E-04 & 2.00 \\ \hline
    & IMEX(3,3;1)& & & IMEX(5,4;1) & \\
    \#Dofs  & \(\text{err}_\infty(0.5)\) & rate & \#Dofs & \(\text{err}_\infty(0.5)\) & rate \\ \hline
    33 & 1.12 E-01 & -- & 49 & 5.46 E-02 & -- \\
    65 & 3.00 E-02 & 1.95 & 97 & 1.43 E-02 & 1.96 \\
    129 & 3.88 E-03 & 2.98 & 193 & 1.83 E-03 & 2.99 \\
    257 & 4.38 E-04 & 3.16 & 385 & 2.01 E-04 & 3.20 \\
    513 & 3.52 E-05 & 3.65 & 769 & 3.98 E-05 & 2.34 \\
    1025 & 3.03 E-06 & 3.55 & 1537 & 3.77 E-06 & 3.41 \\
    \end{tabular}
    \caption{Convergence tables for the KdV equation \eqref{eqn: kdv} on \(D = (-10,10)\). Top left: Euler-IMEX time stepping with CFL = 0.5. Top Right: IMEX(2,2;1) time stepping with CFL = 0.05. Bottom Left: IMEX(3,3;1) time stepping wtih CFL = 0.25. Bottom Right: IMEX(5,4;1) time stepping with CFL = 0.25.}
    \label{tab: imex-convergence}
\end{table}

\noindent The results in Table \ref{tab: imex-convergence} meet our expectations. Euler IMEX is first order accurate, IMEX(2,2;1) is second order, IMEX(3,3;1) is third order, and IMEX(5,4;1) is high order. We thus consider our numerical code verified. 
\subsubsection{Boundary Effects}
The exact solution \eqref{eqn: kdv-solution} is not periodic. However, it is close to zero at the boundaries for the duration of the simulation, and is standard in the literature as a verification test. For the high order methods IMEX(3,3;1) and IMEX(5,4;1), we begin to observe a plateau in the error in the middle column of Table \ref{tab: boundary-effects-wit}. To demonstrate that this phenomenon is because of the boundary and not because of a defect in the method, we extend the domain to \(D = (-20,20)\) and run the test again with same number of degrees of freedom. Extending the domain brings the exact solution closer to zero at the boundary, and so the effects are seen at lower errors. We report the results in table \ref{tab: boundary-effects-wit}. Table \ref{tab: boundary-effects-wit} consists of four subtables. The top left shows the results using the IMEX(3,3;1) method on the domain \(D = (-10,10)\). Here we observe a plateau in the error. The top right shows the results for the IMEX(3,3;1) method, but on the domain \(D = (-20,20)\), and no plateau is observed. The bottom left table shows the results using the IMEX(5,4;1) method on the domain \(D = (-10,10)\); again, we observe a plateau in the error. The bottom right shows the results using the IMEX(5,4;1) method on the domain \(D = (-20,20)\), and we see no plateau in the error. 
    \begin{table}[h!]
        \centering
        \begin{tabular}{c|ccc|ccc}
        && \(D = (-10,10)\) &&& \(D = (-20,20)\) & \\ \hline
        &    \#Dofs  & \(\text{err}_\infty(0.5)\) & rate & \#Dofs & \(\text{err}_\infty(0.5)\) & rate \\ 
        IMEX(3,3;1) & 1025 & 3.03E-06 & -- & 1025 & 5.01E-05 & --  \\   
    & 2049 & 4.55 E-07 & 2.74 &  2049 & 6.65 E-06 & 2.92 \\
      & 4097 & 4.47 E-07 & 0.02 &  4097 & 3.27 E-07 & 4.35\\ \hline
           &    \#Dofs  & \(\text{err}_\infty(0.5)\) & rate & \#Dofs & \(\text{err}_\infty(0.5)\) & rate \\
        IMEX(5,4;1) & 1537 & 3.77 E-06 & -- & 1537 & 3.82 E-05 & -- \\ 
        & 3073 & 4.15 E-07 & 3.19 & 3073 & 4.12 E-06 & 3.22 \\
      & 6145 & 4.42 E-07 & -0.09 & 6145 & 3.82 E-07 & 3.43 \\
        \end{tabular}
        \caption{Comparing convergence rates for the KdV equation \eqref{eqn: kdv} on different domains. Left column: time stepping method. Middle column: errors and rates for domain \(D = (-10,10)\). Right column: errors and rates for domain \(D = (-20,20)\).}
        \label{tab: boundary-effects-wit}
    \end{table}
\par The results found in Table \ref{tab: boundary-effects-wit} meet our expectations of high order convergence, even in the presence of boundary effects. As such, they further serve as verification of the numerical implementation and demonstrate the high order accuracy of the method. 

\subsection{Zabusky-Kruskal}

To demonstrate that our method is capable of solving well-known problems, we apply it to several benchmark cases. The first is the Zabusky-Kruskal experiment (\cite{Zabusky_65}). This simulation uses a KdV-type equation that takes the form: 
\begin{equation}
    \dt u + u\dx u + \delta^2 \dxxx u = 0. \label{eqn: Zabusky-Kruskal-kdv-equation}
\end{equation}
\noindent The authors use \(\delta = 0.022\) with a final time of \(T = \frac{3.6}{\pi}\). The initial condition consists of a cosine wave \(u_0(x) = \cos( \pi x)\) on a domain \(D = (0,2)\). The solution consists of the cosine wave steeping due to the nonlinearity before dispersive waves form into a train of solitons. The original experiment uses an explicit leap-frog method, which is second order in time and in space and stable under a CFL condition like \(\tau = \mathcal{O}(h^3)\).

\par We solve the problem using the Euler-IMEX method with \(\mathbb{P}_1\) elements and 8,193 degrees of freedom. To illustrate the advantage of the high order technique, we also solve the problem using the IMEX(3,3;1) method with \(\mathbb{P}_2\) elements and 1,025 degrees of freedom. 

\par We compare our results to those of Zabusky-Kruskal and a recent simulation done by \cite{Chandramouli_Farhat_Musslimani_22}, which uses a time-dependent spectral normalization (TDSR) method with 256 nodes. Using the software \textit{WebPlotDigitizer} (\cite{WebPlotDigitizer}), we extract data points from the original plot and compare with our solutions at final time \(T = \frac{3.6}{\pi}\). The comparison to Zabusky-Kruskal is given in Figure \ref{fig: zk-vs-imex} and the comparison to \citep{Chandramouli_Farhat_Musslimani_22} is given in Figure \ref{fig: tdsr-vs-imex}.

\begin{figure}[H]
    \centering
    \begin{tikzpicture}
    \begin{axis}[width=0.51\textwidth, xmin =0, xmax = 2, grid=major]
\addplot[thick, purple] table {data_files/imex_data.txt};
        \addplot[thick, teal] table {data_files/zk_data.txt};
\legend{Euler-IMEX, Leap-Frog}
    \end{axis}
\end{tikzpicture}
\begin{tikzpicture}
    \begin{axis}[width=0.51\textwidth,xmin =0, xmax = 2,  grid=major]
        \addplot[thick, teal] table {data_files/zk_data.txt};
         \addplot[thick, red] table {data_files/ryujin_imex_data.txt};
        \legend{Leap-Frog, {IMEX(3,3;1)}}
    \end{axis}
\end{tikzpicture}
    \caption{Comparison of the IMEX method to the Leap-Frog method of \cite{Zabusky_65}. }
    \label{fig: zk-vs-imex}
\end{figure}
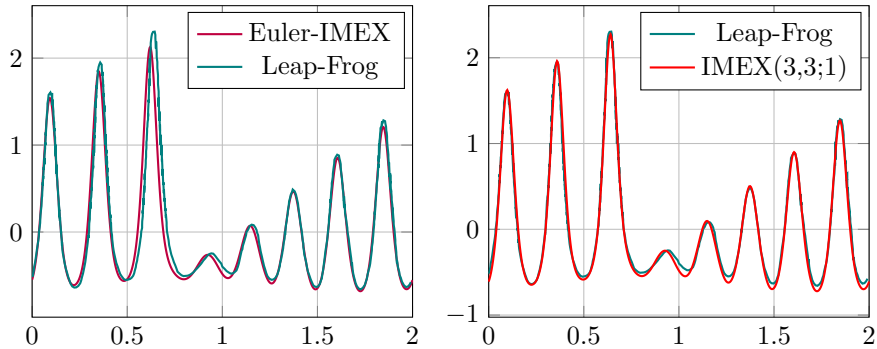
\vspace{-20pt}
\begin{figure}[H]
    \centering
    \begin{tikzpicture}
    \begin{axis}[width=0.52\textwidth, xmin =0, xmax = 2, grid=major]
\addplot[thick, purple] table {data_files/imex_data.txt};
        \addplot[thick, blue] table {data_files/tdsr_data.txt};
\legend{Euler-IMEX, TDSR}
    \end{axis}
\end{tikzpicture}
\begin{tikzpicture}
    \begin{axis}[width=0.52\textwidth, xmin =0, xmax = 2, grid=major]
        \addplot[thick, blue] table {data_files/tdsr_data.txt};
         \addplot[thick, red] table {data_files/ryujin_imex_data.txt};
        \legend{TDSR, {IMEX(3,3;1)}}
    \end{axis}
\end{tikzpicture}
    \caption{Comparison of the IMEX method to the TDSR method of \cite{Chandramouli_Farhat_Musslimani_22}. }
    \label{fig: tdsr-vs-imex}
\end{figure}
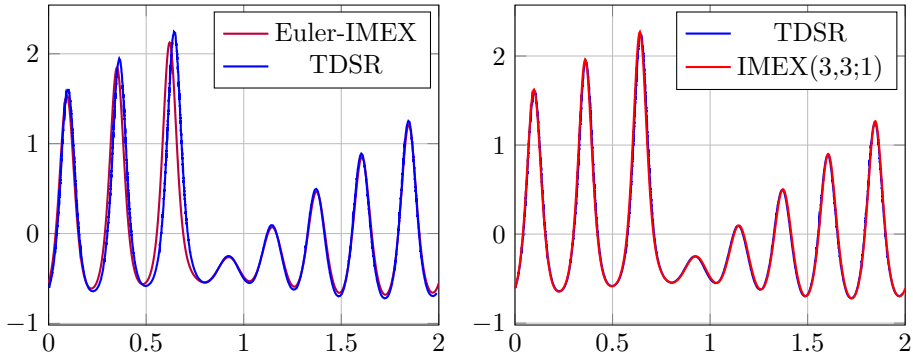

Compared to the leap-frog method, the results are close but differences remain. We believe the discrepancies arise from several factors. First, the IMEX(3,3;1) method is third order accurate, whereas the leap-frog method is only second order. Secondly,  \citep{Zabusky_65} does not provide the mesh or time-step details used to perform the simulation, so it is difficult to provide a full comparison of the discretization, but the CFL condition may have prevented the authors from using a fine mesh. 

 \par When compared to the TDSR method, the Euler-IMEX method shows differences. However, these stem from the low-order accuracy of the method. On the other hand, the IMEX(3,3;1) method shows complete agreement; the small discrepancies at the peaks of the waves are because the TDSR results were obtained using 256 nodes, a coarse mesh compared to the 1,025 degrees of freedom used in the IMEX(3,3;1) solution. 

\subsection{Single Soliton}
The second simulation is the classic single soliton solution with an exact solution \eqref{eqn: kdv-solution}. We use a domain \(D = (-10,10)\) and final time \(T = 2\) with \(\mathbb{P}_1\) finite elements with 1,025 dofs and IMEX(3,3;1) time stepping. The resulting solution matches our expectations: a soliton traveling with constant velocity \(c = 4\) while maintaining it shape and amplitude. Figure \ref{fig: single-soliton} displays the result in a space-time plot.

\begin{figure}[H]
    \centering
\includegraphics[width=0.8\linewidth]{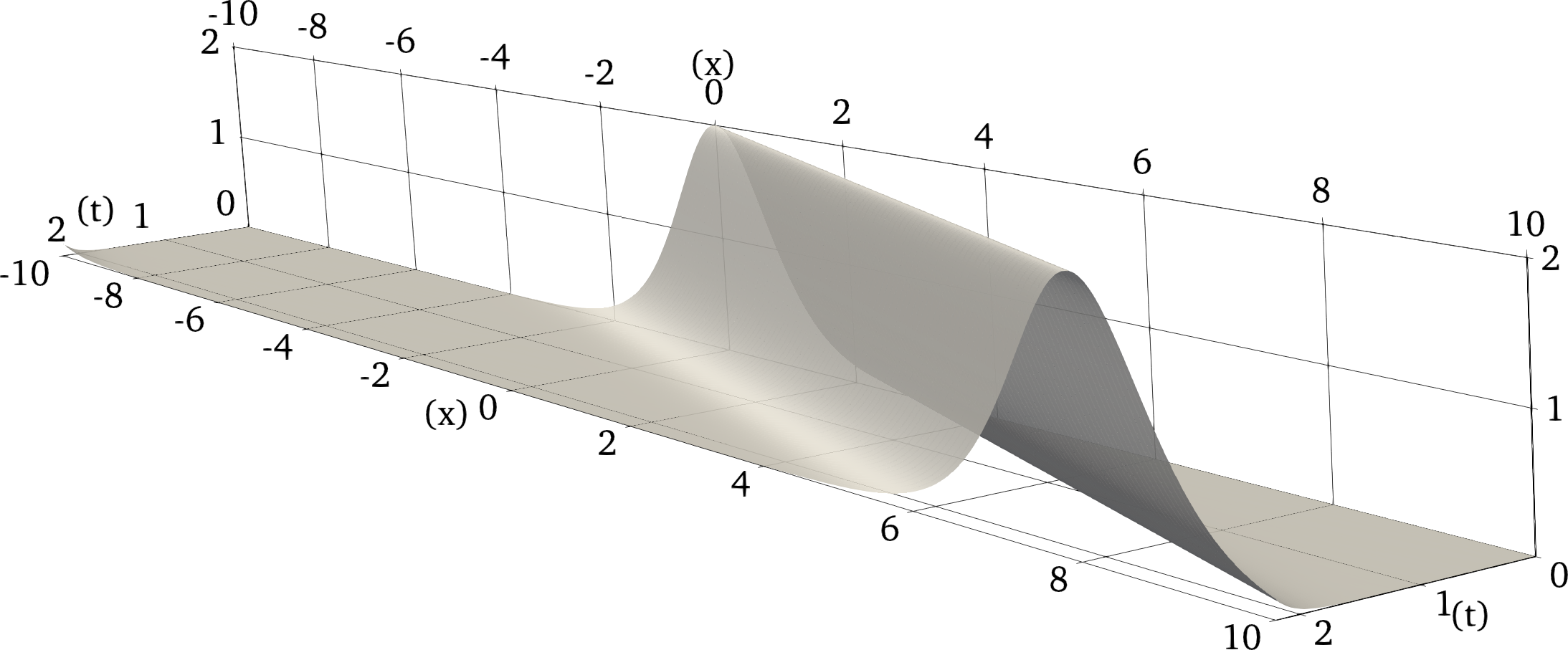}
    \caption{Space time plot for the single soliton solution \eqref{eqn: kdv-solution} of the KdV equation \eqref{eqn: kdv} with initial condition \(u_0(x) = 2\sech^2(x)\). }
    \label{fig: single-soliton}
\end{figure}

\subsection{2-Soliton Interaction}
The third simulation consists of a single wave at \(t = 0\) that separates into two separate solitons of different amplitudes, seen before in \cite{Yoneyama_84, Soomere_2009}. The problem has an exact solution given by: 
\begin{equation}
    u(t,x) = 12 \frac{3 + 4 \cosh(2x - 8t) + \cosh(4x - 64 t)}{(3\cosh(x - 28 t) + \cosh(3x - 36t))^2}. \label{eqn: kdv-2-soliton-solution}
\end{equation}
\noindent First, we present convergence tables for this problem. We solved the KdV equation with \(u_0(x) = 6\sech^2(x)\), which is \eqref{eqn: kdv-2-soliton-solution} when \(t = 0\), on \(D = (-10, 20)\). We use the IMEX(2,2;1) time stepping method with \(\mathbb{P}_1\) elements and CFL = 0.05 on five meshes. We also used the IMEX(3,3;1) time stepping method with \(\mathbb{P}_2\) elements and CFL = 0.25 on five meshes. Table \ref{tab: convergence_soliton_split} displays the results; the results meet expectations. 

\begin{table}[h!]
    \centering
\begin{tabular}{ccc|ccc}
    & IMEX(3,3;1) & & & IMEX(2,2;1) & \\ 
    \#Dofs  & \(\text{err}_{L\infty}(0.5)\) & rate & \#Dofs  & \(\text{err}_{L\infty}(0.5)\) & rate \\ \hline
    257  & 6.15E-01  & --   & 257  & 4.58E-01  & --   \\
    513  & 1.08E-01  & 2.51 & 513  & 1.13E-01  & 2.03 \\
    1025 & 8.28E-03  & 3.72 & 1025 & 2.56E-02  & 2.14 \\
    2049 & 4.95E-04  & 4.07 & 2049 & 6.21E-03  & 2.05 \\
    4097 & 1.52E-05  & 5.02 & 4097 & 1.54E-03  & 2.01 \\ \hline
\end{tabular}
    \caption{Convergence tables for the KdV equation \eqref{eqn: kdv} on \(D = (-10,20)\). Left: IMEX(3,3;1) time stepping with CFL = 0.25 and \(\mathbb{P}_2\) elements. Right: IMEX(2,2;1) time stepping with CFL = 0.05 and \(\mathbb{P}_1\) elements.}
    \label{tab: convergence_soliton_split}
\end{table}

\par Second, we plot the solution to visualize the nonlinear interaction. We use \(\mathbb{P}_1\) finite elements with 1,025 degrees of freedom and IMEX(3,3;1) time stepping on \( D = (-10,10)\). To view the complete interaction, we set the initial time to \(t_0 = -0.5\) and a final time of \(T =0.5\). The resulting solution is presented in Figure \ref{fig: two-soliton}. The left image is a space-time surface plot of the solution; the right is a space-time contour plot of the solution.

\begin{figure}[h!]
\centering\includegraphics[width=0.6\linewidth]{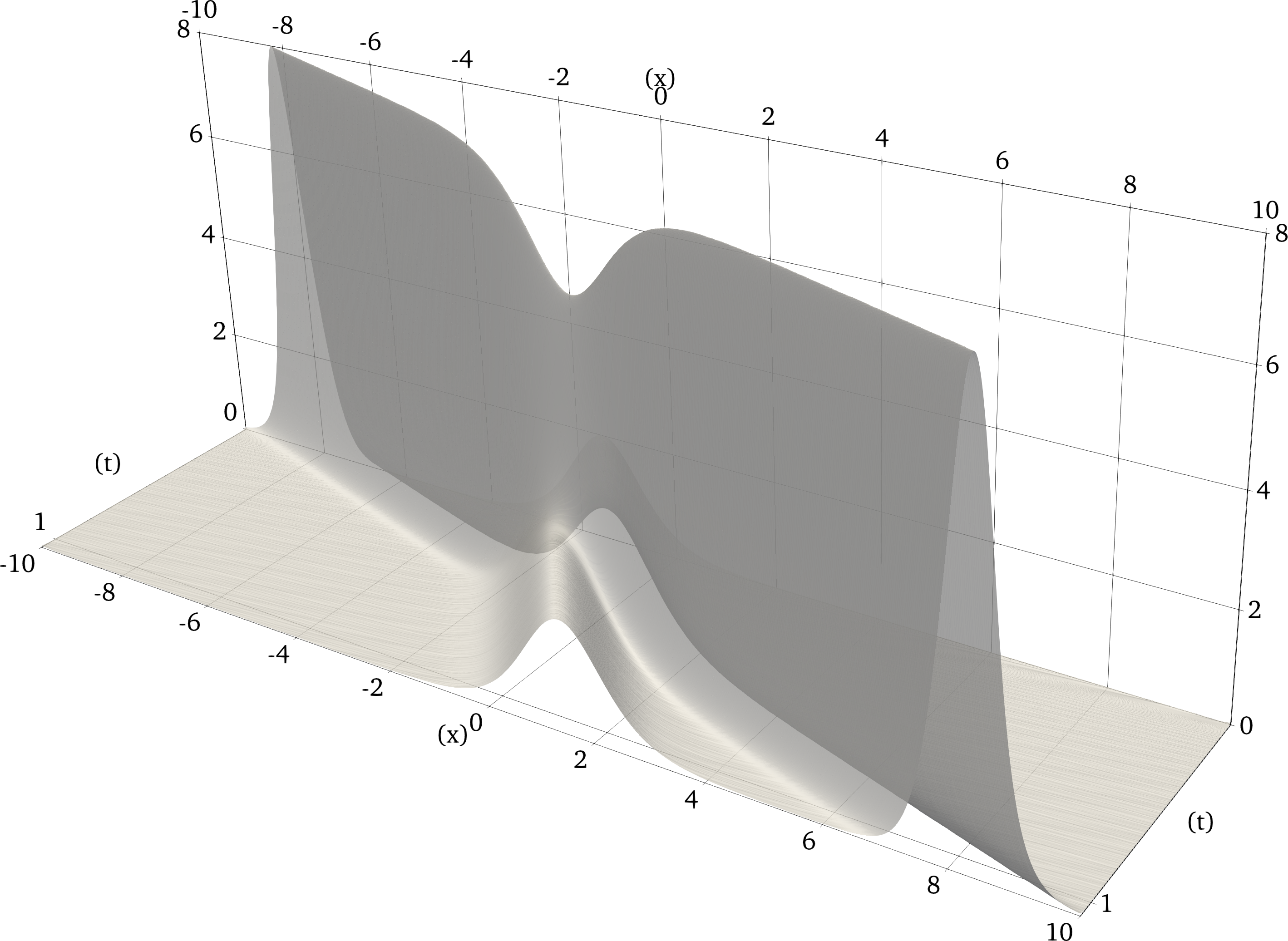}\includegraphics[width=0.25\linewidth, height = 0.6\linewidth]{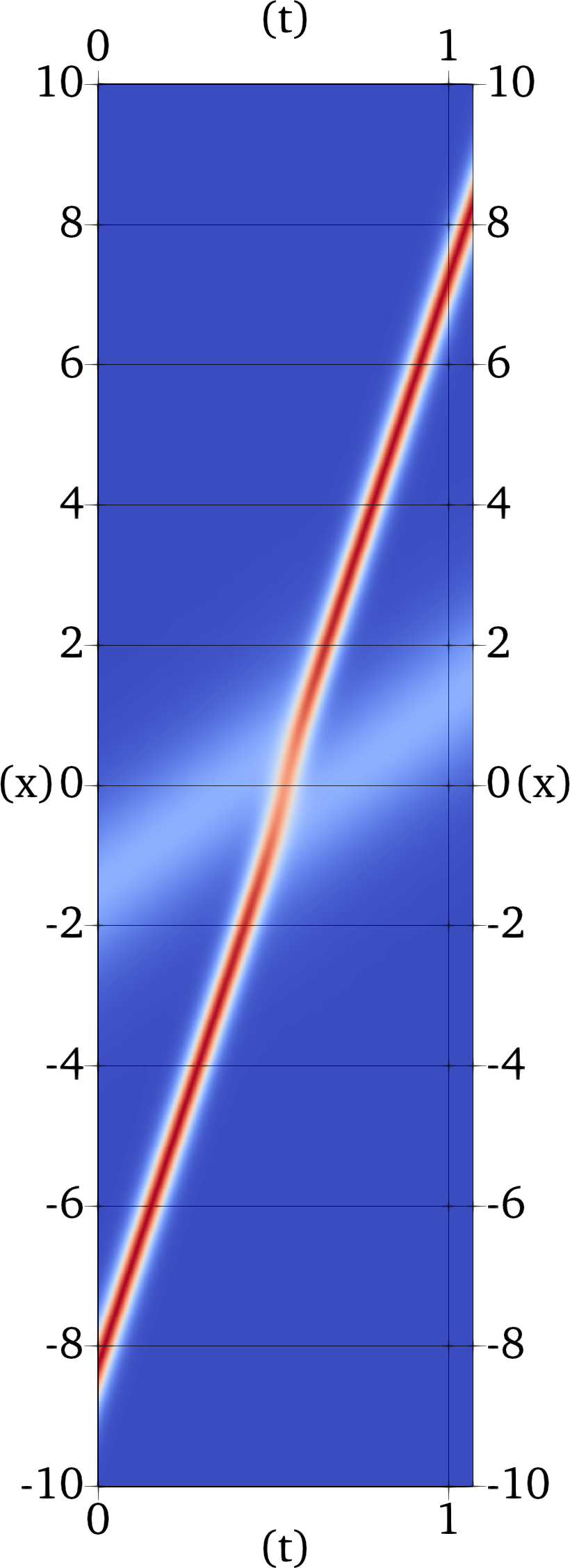}
    \caption{Space-time plots of the two-soliton interacting solution \eqref{eqn: kdv-2-soliton-solution} for the KdV equation \eqref{eqn: kdv}. Left: surface plot. Right: contour plot.}
    \label{fig: two-soliton}
\end{figure}

\par The resulting solution matches the expected behavior. The initial profile consists of a taller, and hence faster, soliton behind a smaller soliton. As time progresses, the larger soliton overtakes the smaller, merging for a moment into a single wave, before before the taller soliton emerges with its original shape, amplitude, and speed. Because of the interaction, the solitons undergo phase shifts, meaning their location is different than if they had traveled without interacting, which is visible in the contour plot of Figure \ref{fig: two-soliton} at approximately \( t= 0.5, ~ x = 0\). 

\subsection{3-Soliton Interaction}
The fourth simulation consists of an initial condition of three super-imposed solitons arranged in descending order of amplitude. The initial condition is given by 
\begin{equation}
    u_0(x) = 4 \sech^2( \sqrt{2}x) + 2 \sech^2( x - 7) + \sech^2\left( \frac{1}{\sqrt{2}} (x - 15)\right),
\end{equation}
\noindent on a domain of \(D = (-5, 45)\) and a final time of \(T = 5\). A similar simulation can be found in \cite{Anco_kdv_solitons}. We use \( \mathbb{P}_1\) finite elements with 2,049 degrees of freedom and IMEX(3,3;1) time stepping method. Figure \ref{fig: soliton-train} displays the results. The left image is a space-time plot of the solution; the right image is a contour space-time plot of the solution. 

\begin{figure}[H]
    \centering
\includegraphics[width=0.6\linewidth, height = 0.4\linewidth]{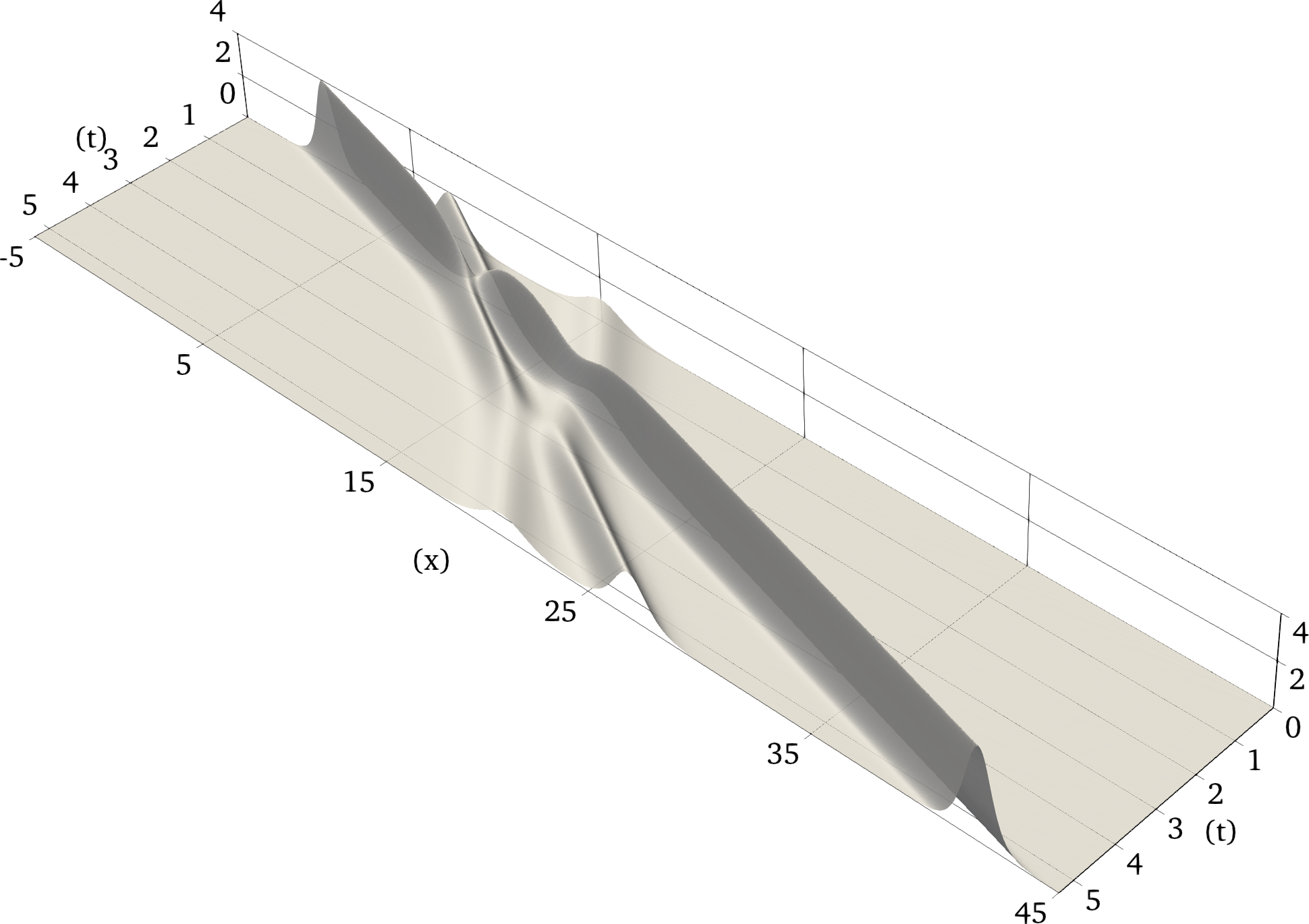}\includegraphics[width=0.25\linewidth, height = 0.5\linewidth]{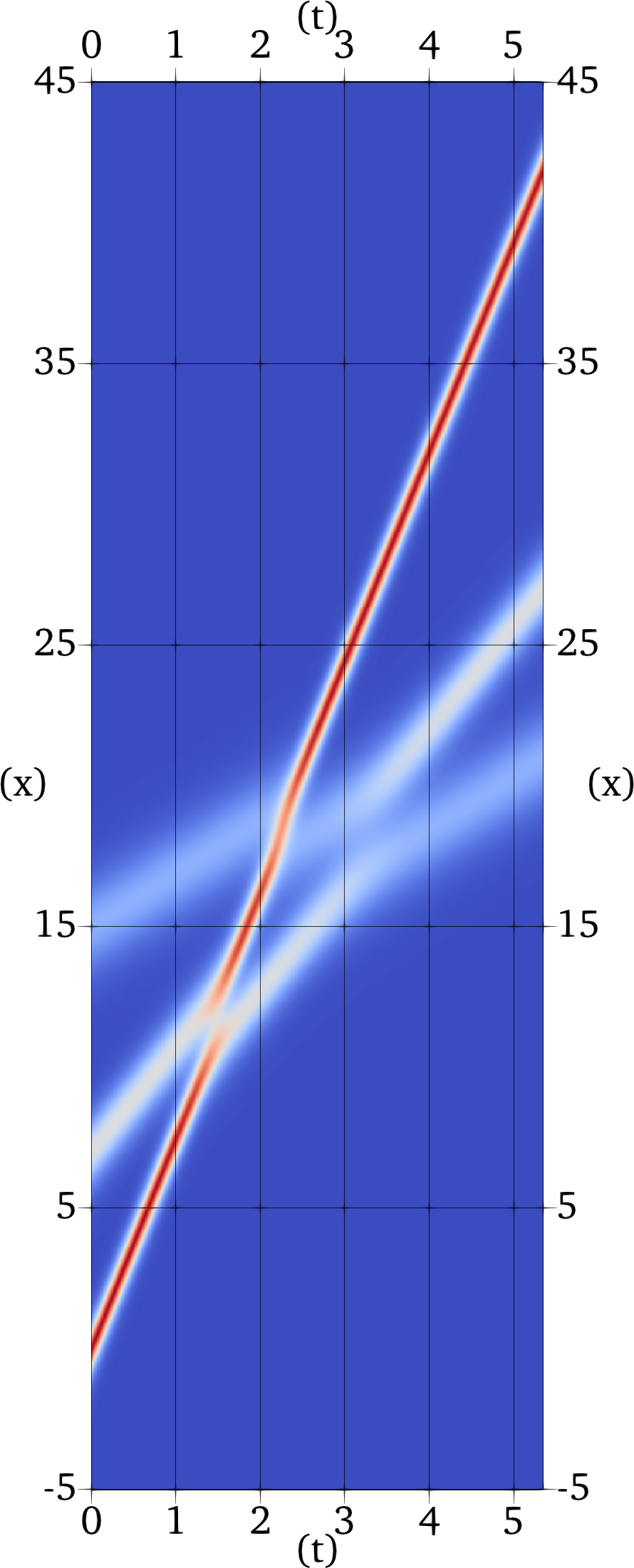}
    \caption{Space-time plots for the three-soliton interaction for the KdV equation \eqref{eqn: kdv}. Left: Surface plot. Right: contour plot.}
    \label{fig: soliton-train}
\end{figure}

We observe three interactions. The first occurs when the tallest soliton interacts with the second tallest and the two ``bump''. The three waves remain visible, but the middle soliton grows in amplitude and thereby begins to travel faster while the left soliton loses amplitude and slows. This interaction can be seen at approximately \(t = 1.5,~ x = 11\) in the contour plot. The second interaction occurs when the now tallest middle soliton interacts with the smallest. The two briefly merge into a single wave before the larger passes through, maintaining its shape and speed but undergoing a phase shift, similar to the 2-soliton test. This can be seen at at approximately \(t = 2.5, ~ x = 18\) in the contour plot. The third interaction occurs when the now second tallest soliton reaches the smallest, and a ``bump'' interaction can be observed. This can be seen at approximately \(t = 3.5,~ x = 20\) in the contour plot.

\section{Conclusion}
An IMEX scheme using \(H^1\)-conforming finite elements has been introduced for generalized KdV equations. The method is provably stable up to a CFL condition \(\tau = \mathcal{O}(h)\) and mass-conservative. High-order techniques have been introduced and analyzed. The numerical simulations demonstrate both the low and high order capabilities and robustness of the method. The next extension of this work is to apply the approach to dispersive systems, such as the Serre-Green-Naghdi shallow water equations (\cite{Guermond_Kees_Popov_Tovar22_1}). Solving higher-order models, such as the Kawahara equation \cite{Biswas_09} is another direction for future research. Another avenue would be to modified the method to account for Boussinesq type models, which often include terms like \(\dt \dxx u\); see \cite{Benjamine_Bona_Mahony_97}. 

\section*{Acknowledgments}
The author would like to thank Jean-Luc Guermond for guidance and advice in service of this work, as well as Matthias Maier for assistance in implementing the method into \verb|ryujin|. Also, the author would like to thank Eric J. Tovar for helpful discussions and insights and Jordan Hoffart for helpful commentary on this manuscript.



\end{document}